\documentclass[letterpaper, 10 pt, conference]{ieeeconf}  

\usepackage{algorithm}
\usepackage{algpseudocode}
\usepackage{array}
\usepackage{cite}

\IEEEoverridecommandlockouts                              

\title{\LARGE \bf
	Distributed Nash Equilibrium Seeking in Non-Monotone Games over the Simplex}
\author{Tatiana Tatarenko, S. Rasoul Etesami
	\thanks{T. Tatarenko is with the Intelligent Systems and Robotics Lab at the TU Darmstadt, Germany (e-mail: \tt\small tatiana.tatarenko@tu-darmstadt.de).}
    \thanks{S.R. Etesami is with the Department of Industrial and Systems Engineering, Coordinated Science Laboratory, University of Illinois Urbana-Champaign, USA  (e-mail: \tt\small etesami1@illinois.edu).}}

\usepackage{amsfonts,amsmath,amssymb,bm,amsthm}
\usepackage{color,cases,float}
\usepackage{graphicx}
\usepackage{hyperref}
\usepackage{psfrag}
\usepackage{eucal}
\usepackage{algorithm}
\usepackage{overpic}
\usepackage{subfigure}
\usepackage{url}
\usepackage[all,cmtip]{xy}
\usepackage{algorithm}

\definecolor{darkblue}{rgb}{0.0,0.0,0.6}
\hypersetup{colorlinks,breaklinks,
	linkcolor=darkblue,urlcolor=darkblue,
	anchorcolor=darkblue,citecolor=darkblue}
\newtheorem{assumption}{Standing Assumption}
\newtheorem{assum}{Assumption}
\newtheorem{definition}{Definition}
\newtheorem{lem}{Lemma}

\newtheorem{prop}{Proposition}
\newtheorem{theorem}{Theorem}
\newtheorem{remark}{Remark}

\newtheorem*{example*}{Example}

\newcommand{\R}{\mathrm{Re}}
\newcommand{\Proj}{\mathrm{Proj}}

\def\R{\mathbb{R}}

\def\hD{\hat{\Delta}^{\tau}}
\def\tx{\tilde{x}}

\def\R{\mathbb{R}}

\def\BibTeX{{\rm B\kern-.05em{\sc i\kern-.025em b}\kern-.08em
		T\kern-.1667em\lower.7ex\hbox{E}\kern-.125emX}}

\def\BibTeX{{\rm B\kern-.05em{\sc i\kern-.025em b}\kern-.08em
    T\kern-.1667em\lower.7ex\hbox{E}\kern-.125emX}}
\begin{document}

\maketitle

\begin{abstract}
In this work, we present a novel characterization of approximate Nash equilibria in a class of convex games over the simplex. To achieve this, we regularize the utility functions using the Shannon entropy term, connect the solutions to the regularized game with the set of Nash equilibria, and formulate a multi-objective optimization problem to solve the regularized game. Based on the obtained properties of the stationary points in this optimization problem, we formulate two distributed heuristic algorithms to compute an approximate Nash equilibrium of the original game.

\end{abstract}

\vspace{-0.5cm}
\section{Introduction}
Game-theoretic optimization addresses a specific class of problems arising in multi-agent systems, where each agent, called a player, chooses its action from a local action set to maximize its local utility function. These utility functions are coupled through the decision variables (actions) of all players in the system. Applications of game-theoretic optimization can be found in areas such as electricity markets, communication networks, autonomous driving systems, and smart grids \cite{Alpcan2005, GTM, BasharSG}. Solving a game-theoretic optimization problem involves finding a stable joint action, called a \emph{Nash equilibrium}, from which no agent has any incentive to unilaterally deviate. 

In the case of \emph{convex games}, Nash equilibria are equivalent to solutions of a specific variational inequality defined for the game's pseudo-gradient over the joint action set \cite{FaccPang1}. Such variational inequalities can be efficiently solved under the assumption that the pseudo-gradient possesses monotonicity properties~\cite{LanExtrapolation, Nesterov}, or when a \emph{Minty solution} exists in the game~\cite{He}. In recent years, much work has been devoted to the development of \emph{distributed procedures} that leverage only local agents' information to achieve a solution in such games~\cite{gao2021second, pethick2023solving, TCNS24, ecc24}.

While monotone games or games with a Minty solution can be efficiently solved either centrally or in a distributed setting, solution approaches for non-monotone convex games are not well-studied. Moreover, it is known that the problem of approximating a mixed Nash equilibrium in a general two-player, non-zero sum normal form game is PPAD-hard (see~\cite{Chen})\footnote{PPAD-hard problems belong to the class PPAD (Polynomial Parity Argument, Directed version) introduced by Papadimitriou in 1991. For more details on this class, see~\cite{Chen} and the references therein.}. The recent paper~\cite{anagnostides2022optimistic} studies the \emph{optimistic mirror descent} algorithm and demonstrates that it either converges to a Nash equilibrium or that the average correlated distribution of play is a strong \emph{coarse correlated equilibrium}. Some works have addressed a specific subclass of such games, namely \emph{Rank-1 games}, where a Nash equilibrium can be solved using polynomial-time algorithms~\cite{patris2024learning}.

Despite the well-established impossibility results, current works continue to seek heuristics for circumventing computational obstacles in general convex games. In particular, the work~\cite{pmlr-v97-raghunathan19a} presents an optimization approach based on the specifically defined Nikaido-Isoda function to calculate stable states in unconstrained games. The paper~\cite{gemp2024approximating}, in turn, formulates a non-convex stochastic optimization problem to approximate a Nash equilibrium in normal-form games. The main idea in~\cite{gemp2024approximating} involves the regularization of locally convex utility functions using the Shannon entropy term, where the resulting regularized game possesses Nash equilibria. Moreover, each Nash equilibrium in the regularized game is an approximate solution to the original one and belongs to the interior of the joint action set. This condition implies a reformulation of the corresponding variational inequality as a nonlinear system of equations, which leads to the formulation of a non-convex optimization problem. 

\subsection{Contributions and Organization}
Inspired by the idea in \cite{gemp2024approximating}, we continue \emph{characterizing approximate Nash equilibria in a class of convex games over the simplex} and formulate novel \emph{distributed heuristic algorithms} to identify such points. Our contributions are summarized as follows:
 
\begin{itemize}
    \item We formulate a multi-objective optimization problem in the regularized game (see~\eqref{eq:MOprob}), where each player controls a number of objective functions, and provide a result characterizing approximate Nash equilibria ($\epsilon$-NE) as points at which the gradients of the objectives are equal to zero (Lemmas~\ref{lem:MO}-\ref{lem:lem1});
    \item We further analyze the set of common stationary points of the introduced multi-objective optimization problem and demonstrate that this set is equal to the set of $\epsilon$-NE in the game, given an appropriate choice of the regularization parameter (Lemma~\ref{lem:stationary});
    \item Based on the results above, we propose two heuristics aimed at approximating Nash equilibria in the game under consideration.
\end{itemize}

The paper is organized as follows. Section~\ref{sec:II} defines the problem of game-theoretic optimization in a class of convex games over the simplex, introduces the notion of an approximate Nash equilibrium ($\epsilon$-NE), and presents important preliminary results on the regularized game obtained in the work~\cite{gemp2024approximating}. Section~\ref{sec:III} formulates a multi-objective optimization problem, whose solution set is equal to the set of approximate Nash equilibria, and analyzes the stationary points of the objective functions. Section~\ref{sec:IV} leverages this analysis and presents two heuristic algorithms for computing an approximate Nash equilibrium. Section~\ref{sec:V} presents some simulation results, whereas Section~\ref{sec:VI} concludes the paper.

\subsection{Notations}

The set $\{1, \ldots, N\}$ is denoted by $[N]$. For any function $f(x) : K \to \mathbb{R}$ and a column vector $x = (x_1, \ldots, x_N)$, where $x_i \in \mathbb{R}^{n_i}$, we define the partial derivative $\nabla_{i} f(x) = \frac{\partial f(\mathbf{x})}{\partial x_i}$ as the derivative with respect to the $x_i$-th coordinate of the vector $x = (x_1, \ldots, x_N) \in \mathbb{R}^{\sum_{i=1}^N n_i}$. We use $\Proj_{\Omega}(v)$ to denote the projection of a vector $v \in \mathbb{R}^d$ onto a set $\Omega \subseteq \mathbb{R}^d$. We denote the dot product in $\mathbb{R}^n$ by $\langle \cdot, \cdot \rangle$.

\section{Optimization-based Reformulation for Concave Games over the Simplex}\label{sec:II}
We consider a noncooperative game denoted by $\Gamma = \Gamma(N, \{\Delta_i\}, \{f_i\})$, consisting of $N$ players, each with their local action set $\Delta_i$ and local utility function $f_i : \Delta = \Delta_1 \times \dots \times \Delta_N \to \mathbb{R}$.
We denote by $x_i = (x_i^1, \ldots, x_i^{n_i+1})$ an action of player $i$ chosen from $\Delta_i$.
We let $x = (x_i, x_{-i})$ denote an action profile from the joint action set, where $x_{-i}$ represents the joint action of all players except player $i$.
In this work, we focus on games where each action set $\Delta_i$ and utility function $f_i$ possess the following structure.

\begin{assumption}\label{assum:convex1}
Each action set $\Delta_i$ is the $(n_i+1)$-dimensional simplex defined as follows: 
\[
\Delta_i = \Big\{ x_i = (x_i^1, \ldots, x_i^{n_i+1}) \in \mathbb{R}_+^{n_i+1} : \ \sum_{j=1}^{n_i+1} x_i^j = 1
\Big\}.\]
Each utility function $f_i(x)$, for $i \in [N]$, is a polynomial function of degree 2 and has the form
\[
f_i(x) = \sum_{k \neq i} x_i^T Q_{ik} x_k + \sum_{j=1}^N r_j^T x_j,
\]
where each $Q_{ik}$ is an $(n_i+1) \times (n_k+1)$ dimensional matrix with the $(j,l)$-entry being $q_{ik}^{jl}$, for $j = 1, \ldots, n_i+1$ and $l = 1, \ldots, n_k+1$, and $r_i = (r_i^1, \ldots, r_i^{n_i}) \in \mathbb{R}^{n_i}$ is a vector. 
\end{assumption}
\begin{example*} The following games satisfy Assumption~\ref{assum:convex1} above:
\begin{enumerate}
    \item Any game with two players and finite action sets in mixed strategies is an example of a game over the simplex that satisfies Assumption~\ref{assum:convex1}.
    \item Another example of a game over the simplex where Assumption~\ref{assum:convex1} holds is stable matching games in mixed strategies with \emph{waiting list feedback}, as presented in \cite{arxiv_Matching} (see equation (30) therein).
\end{enumerate}
 \end{example*}

 \begin{definition}\label{def:NE}
A joint action profile $x^*\in\Delta$ is a \emph{Nash equilibrium} (NE) in the game $\Gamma$ iff $f_i(x^*_i,x^*_{-i})\ge f_i(x_i, x^*_{-i})$, which holds for any $x_i\in\Delta_i$ and $i\in[N]$. A joint action  $x^*\in\Delta$ is an \emph{$\epsilon$-Nash equilibrium} ($\epsilon$-NE) in the game $\Gamma$ iif $f_i(x^*_i,x^*_{-i})\ge f_i(x_i, x^*_{-i}) - \epsilon$, which holds for any $x_i\in\Delta_i$ and $i\in[N]$.
\end{definition}

In this work, we are interested in developing distributed algorithms to obtain an $\epsilon$-NE, also called approximate Nash equilibrium, in the class of games satisfying Assumption~\ref{assum:convex1}. We note that, although the utility functions have a quadratic form and belong to the class of concave games,\footnote{A continuous action game is said to be concave (convex) if the payoff of each player is a concave (convex) function of their own action \cite{rosen1965existence}.}  we do not assume anything about the positive definiteness of the matrices. As a result, the games under consideration are neither monotone nor belong to the class of potential games, which presents one of the major challenges in obtaining their NE points. 

 \begin{remark}\label{rem:compl}
     It is important to note that since the class of two-player nonzero-sum normal-form games in mixed strategies is a special case of the games under consideration, we cannot expect an efficient calculation of the $\epsilon$-NE points, as the latter problem is known to be PPAD-hard (see~\cite{Chen}). Thus, the main goal of this paper is to present a novel characterization of the $\epsilon$-NE points for games that satisfy Assumption \ref{assum:convex1}, based on which heuristic algorithms can be developed.  
 \end{remark}

Let $F(x)$ denote the pseudo-gradient of the game $\Gamma$, i.e., 
\[F(x) = [\nabla_{1}f_1(x), \ldots, \nabla_{N}f_N(x)]^T\in\R^d,\]
where $d = \sum_{i=1}^N(n_i+1)$.
To find a NE in a game satisfying Assumption~\ref{assum:convex1}, one can aim at solving the following variational inequality (see \cite[Corollary 2.2.5]{FaccPang1} for the equivalence between NE points in convex games and solutions to variational inequalities): 
\begin{align}\label{eq:VI0}
    \mbox{Find $x^*\in\R^d$: }\langle F(x^*),x-x^*\rangle\le 0\, \mbox{ for any $x\in\Delta$}.
\end{align}
We note that a solution to the variational inequality~\eqref{eq:VI0} always exists, as $\Delta$ is a compact set. We refer to~\cite{gao2021second, pethick2023solving, TCNS24, ecc24} for different gradient-based solution approaches for solving \eqref{eq:VI0} under some specific assumptions regarding the mapping $F(\cdot)$, such as (strong) monotonicity or existence of a so called Minty solution.

In convex games where the pseudo-gradient does not possess one of the properties above, the gradient-based procedure may diverge~\cite{GRAMMATICO2018186}. To rectify this issue, one can turn attention to an optimization-based approach for NE computation. However, all known reformulations for games with constraints are represented as non-convex optimization problems, and convergence to a stationary point of the objective function can only be guaranteed, which does not always coincide with the set of NE points~\cite{gemp2024approximating, pmlr-v97-raghunathan19a}. In this work, we will present a non-convex multi-objective optimization problem where common stationary points of the objective functions are the set of $\epsilon$-NE in $\Gamma$.    

\subsection{Quantal Response Equilibria}

Following the approach presented in \cite{gemp2024approximating}, we aim to reformulate the NE seeking problem as a specific optimization problem. As discussed in \cite{gemp2024approximating}, to capture all equilibria (i.e., interior NE where \( x^* \in \text{int} \Delta \) and boundary NE where \( x^* \in \partial \Delta \)), we focus on the \emph{quantal response equilibria} defined as follows.

\begin{definition}\label{def:QRE}
A joint action $x_{\tau}^*\in\Delta$ is a \emph{quantal response equilibrium} (QRE) in the game $\Gamma$ iff it is a NE in the game $\Gamma^{\tau} = \Gamma^{\tau} (N,\{\Delta_i\},\{f^{\tau}_i(x) =f_i(x) - \tau \sum_{j=1}^{n_i+1}x_i^j\ln(x_i^j)\})$, where $\tau>0$.
The set of quantal response equilibria in the game $\Gamma$ is denoted by QRE($\tau$).
\end{definition}

To define a QRE, one perturbs the initial utility functions \( f_i \), \( i \in [N] \), by the Shannon entropy term
\begin{align}\label{eq:Shan}
S(x_i) = -\sum_{j=1}^{n_i+1} x_i^j \ln(x_i^j),
\end{align}
which is a strongly concave function over \( \Delta_i \). Thus, the perturbed game \( \Gamma^{\tau} \) admits a NE, and moreover, any NE in \( \Gamma^{\tau} \) belongs to the set QRE(\( \tau \)).

Next, we state a result from \cite{gemp2024approximating}, which shows that any \( x \in \text{QRE}(\tau) \) belongs to \( \text{int} \Delta \), specifies a lower bound for the coordinates of \( x \), and demonstrates that \( x \) must be an \( \epsilon \)-NE in \( \Gamma \) for a specific choice of the parameter \( \tau \) that satisfies \( \epsilon \geq \tau \max_i \{\ln (n_i + 1)\} \). The proof of the following lemma can be found in \cite[Lemmas 11 and 12]{gemp2024approximating}.

\begin{lem}\label{lem:epsApp} The following statements hold for the game $\Gamma$:
\begin{enumerate}
    \item Given a positive $\tau>0$, we have that  $\mbox{QRE}(\tau)\subseteq\mbox{int}\Delta$.
    \item Let $\tau \le \frac{\epsilon}{\max_i\{\ln (n_i+1)\}}$ for any fixed $\epsilon>0$.
    Then, any $x\in \mbox{QRE}(\tau)$ is an $\epsilon$-NE in the game $\Gamma$. 
    \item Each coordinate of $x\in \mbox{QRE}(\tau)$ is lower-bounded by the value $x_{\min}=\frac{\exp(-\delta_f/\tau)}{\max_i (n_i+1)}$, where $\delta_f = \max_{i\in[N]}(\max_{x\in\Delta}f_i(x) - \min_{x\in\Delta}f_i(x))$.
\end{enumerate}
\end{lem}


\subsection{Optimization problem for QRE seeking}
Now, we intend to formulate an optimization problem whose solutions capture the set \(\text{QRE}(\tau)\) in \(\Gamma\). According to Lemma~\ref{lem:epsApp}, this set coincides with the \(\epsilon\)-NE in \(\Gamma\), where \(\epsilon\) can be chosen arbitrarily small, given an appropriate setting for the parameter \(\tau\).

For the sake of analytical convenience, we reformulate the action sets and utility functions in the game \(\Gamma\) by expressing the coordinates \(x_{n_i+1}\), \(i \in [N]\), as follows:
\begin{align}\label{eq:tr}
x_{n_i+1} = 1 - \sum_{j=1}^{n_i}x_j \, \mbox{ for each $i\in[N]$}.
\end{align}
Therefore, in the subsequent analysis, we consider the games \(\Gamma^t\) and \(\Gamma^{\tau,t}\), which are obtained by applying the \emph{transformation}~\eqref{eq:tr} to the games \(\Gamma\) and \(\Gamma^{\tau}\), respectively. In particular, the action sets in these transformed games are given by
\begin{align}\label{eq:set_t}
    &\Delta^{t} = \times_{i=1}^N\Delta_i^{t},\cr
    &\Delta_i^{t} = \Big\{x_i=(x_i^1,\ldots,x_i^{n_i})\in\R^{n_i}_+:
     \sum_{j=1}^{n_i} x_i^j\le1 \Big\}.
\end{align}
Moreover, each utility function in the game $\Gamma^{\tau,t}$ has the form 
\begin{align}\label{eq:ut_t}
    f^{\tau,t}_i(x) = f_i^t(x)&-\tau \sum_{j=1}^{n_i}x_i^j\ln(x_i^j) \cr
    &-\tau(1-\sum_{j=1}^{n_i}x_i^j)\ln(1-\sum_{j=1}^{n_i}x_i^j),
\end{align}
for $x\in\Delta^t$, where \(f_i^t\) is the utility function in the game \(\Gamma^t\), obtained from the function \(f_i\) by applying the transformation~\eqref{eq:tr}. 
Taking into account Assumption~\ref{assum:convex1} and the linearity of the transformation~\eqref{eq:tr}, we conclude that the functions \(f_i^t : \Delta^t \to \mathbb{R}\), \(i \in [N]\), have the following structure:\footnote{The utility structures in \eqref{eq:ut_t1} are valid up to a constant shift, which does not affect the subsequent optimizations.}
\begin{align}\label{eq:ut_t1}
    f^t_i(x) = \sum_{k\ne i} x_i^T\hat Q_{ik}x_k  + \sum_{j=1}^N \hat r_j^Tx_j,
\end{align}
where each $\hat Q_{ik}$ is an $n_i\times n_k$ dimensional matrix with the $(j,l)$-entry equal to $\hat q_{ik}^{jl}$, $j=1,\ldots, n_i$, $l=1,\ldots, n_k$, ($\hat q_{ik}^j$ is the $j$th row in the matrix $\hat Q_{ik}$) and $\hat r_i = (\hat{r}_i^1,\ldots,\hat{r}_i^{n_i}) \in \R^{n^i}$ is a vector.

Given the equivalence between \(\Gamma\), \(\Gamma^{\tau}\), and \(\Gamma^t\), \(\Gamma^{\tau,t}\), henceforth, we focus on the solutions to the game \(\Gamma^{\tau,t}\), which are approximate NE in \(\Gamma^t\) and, thus, in the original game \(\Gamma\).  
For this purpose, let us consider the \emph{pseudo-gradient} \(F^{\tau,t} : \Delta^t \to \mathbb{R}^{\sum_{i \in [N]} n_i}\) of the game \(\Gamma^{\tau,t}\): 
\[F^{\tau,t}(y) =\big(\nabla_1f^{\tau,t}_1(y),\ldots, \nabla_Nf^{\tau,t}_N(y)\big). \]
Leveraging again the result of ~\cite[Corollary 2.2.5]{FaccPang1}, we conclude that $x_{\tau}^{*,t}$ is a NE in $\Gamma^{\tau,t}$ iff 
 \begin{align}\label{eq:VI}
     \langle F^{\tau,t}(x_{\tau}^{*,t}), y - x_{\tau}^{*,t}\rangle \le 0 \, \mbox{ for any $y\in\Delta^t$}.
 \end{align}
 
\begin{lem}\label{lem:solutionsQRE}  
The point $x = x_{\tau}^{*,t} \in \Delta^t$ is a NE of the game $\Gamma^{\tau,t}$ if and only if
$\|F^{\tau,t}(x)\|^2 = \sum_{i=1}^N \|\nabla_i f^{\tau,t}_i(x)\|^2 = 0$. Moreover, any such NE $x \in \Delta^t$ that is lifted to $\Delta$ by setting $x_{n_i+1} = 1 - \sum_{j=1}^{n_i} x_i^j$ for each $i \in [N]$ belongs to QRE$(\tau)$ in the game $\Gamma$.   
\end{lem}
\begin{proof} 
 On the one hand, according to Lemma~\ref{lem:epsApp}, any point from the set QRE($\tau$) belongs to $\mathrm{int}\Delta$.  
On the other hand, the set QRE($\tau$) is the set of NE in the game $\Gamma^{\tau}$. Taking into account the equivalence between $\Gamma^{\tau}$ and $\Gamma^{\tau,t}$ under the coordinate transformation~\eqref{eq:tr}, as well as the definition of $\Delta^t$ (see~\eqref{eq:set_t}), we conclude that any NE $x = x_{\tau}^{*,t}$ of the game $\Gamma^{\tau,t}$ belongs to $\mathrm{int}\Delta^t$ and solves the variational inequality~\eqref{eq:VI}. Thus, $x \in \Delta^t$ is a NE of the game $\Gamma^{\tau,t}$ if and only if $F^{\tau,t}(x) = \boldsymbol{0}$, which is equivalent to the condition $\|F^{\tau,t}(x)\|^2 = \sum_{i=1}^N \|\nabla_i f^{\tau,t}_i(x)\|^2 = 0.$ The second statement of the lemma holds again according to the definition of the game $\Gamma^{\tau,t}$ and its equivalence to $\Gamma^{\tau}$, given the transformation~\eqref{eq:tr}.
\end{proof}

The lemma above allows us to formulate the following constrained optimization problem, whose (lifted) optimal solutions coincide with the set QRE($\tau$) in the game $\Gamma$:
\begin{align}\label{eq:problem0}
&\min_{x\in\Delta^t}\sum_{i=1}^N\|\nabla_if^{\tau,t}_i(x)\|^2,
\end{align}
where the functions $f^{\tau,t}_i(x)$ are defined by~\eqref{eq:ut_t}.
In particular, according to Lemma~\ref{lem:epsApp}, the set of optimal solutions of \eqref{eq:problem0} represents approximate NE points for the original game $\Gamma$. We note that the above optimization problem always has a solution, as $\Delta^t$ is a compact set. Moreover, when $\tau = 0$, the problem~\eqref{eq:problem0} is a convex optimization problem. However, in the case $\tau = 0$, we cannot guarantee equivalence between the solution set in~\eqref{eq:problem0} and the NE points in $\Gamma$. If $\tau = 0$, the problem~\eqref{eq:problem0} can only capture NE points in $\Gamma$ from $\mathrm{int}\Delta$, if such points exist. This is the main reason we perturb the utility functions by the Shannon entropy term~\eqref{eq:Shan}, multiplied by a non-zero $\tau$, to be able to calculate an approximate NE in the game $\Gamma$, even in the absence of an interior NE. As a result of such perturbation, the optimization problem~\eqref{eq:problem0} becomes non-convex. In the next section, we present a reformulation of this non-convex optimization problem in terms of multi-objective non-convex optimization one in which the common stationary points of objective functions represent approximate Nash equilibria in $\Gamma$, given an appropriate choice of the regularization parameter $\tau$. This reformulation will allow for the development of distributed heuristic algorithms to obtain an $\epsilon$-NE.

\section{Multi-Objective Problem Formulation}\label{sec:III}
Let us fix a small $\epsilon>0$ and choose $\tau \le \frac{\epsilon}{\max_i\{\ln (n_i+1)\}}$. We start by defining the following constrained set for each player $i\in[N]$: 
\begin{align}\label{eq:extended_Di}
\hD_i = \Bigg\{x_i\in\R^{n_i}:\ \ &x^j_i\ge \frac{e^{-1/\tau^2}}{\max_i n_i+1} \ \forall j\in [n_i], \cr
&\sum_{j=1}^{n_i} x_i^j\le1-\frac{e^{-1/\tau^{1.5}}}{\max_i n_i + 1}\Bigg\}.     
\end{align}
According to Lemma~\ref{lem:epsApp} and because $\delta_f$ is a positive constant, we have $x_{\min}=\frac{e^{\Omega(-1/\tau)}}{\max_i (n_i+1)}$ for any $x\in \mbox{QRE}(\tau)$. Since by Lemma \ref{lem:solutionsQRE} any NE of $\Gamma^{\tau,t}$ belongs to $\mbox{QRE}(\tau)$, using Lemma~\ref{lem:epsApp} (part 3), we conclude that for any $\tau$ satisfying
\begin{align}\label{eq:tau1}
\tau\le\min\left\{\frac{1}{\delta_f}, \frac{1}{\delta^2_f}\right\},
\end{align}
the set of NE in $\Gamma^{\tau,t}$ is contained in 
\begin{align}\label{eq:extended_D}
\hD=\times_{i=1}^N\hD_i.     
\end{align}



We focus on each term $\|\nabla_if^{\tau,t}_i(x)\|^2$ in the objective function of \eqref{eq:problem0}. Let us use the notation $l_i^{\tau}(x) = \|\nabla_if^{\tau,t}_i(x)\|^2$. Using~\eqref{eq:ut_t}, we can write
\begin{align}\label{eq:l_i}
    &l_i^{\tau}(x)= \|\nabla_if^{\tau,t}_i(x)\|^2 \cr
    &= \Big\|\nabla_if^t_i(x) - \tau (\boldsymbol{1}_{n_i}\!+\!\ln x_i)+\tau \Big(1+\ln\big(1\!-\!\sum_{\ell=1}^{n_i}x_i^{\ell}\big)\Big)\boldsymbol{1}_{n_i}\Big\|^2 \cr
    &= \sum_{j=1}^{n_i} \Big(\frac{\partial f^t_i(x)}{\partial x_i^j} - \tau\ln x_i^j + \tau\ln\big(1-\sum_{\ell=1}^{n_i}x_i^{\ell}\big) \Big)^2\cr
    &= \sum_{j=1}^{n_i}l_{i,j}^{\tau}(x),
\end{align}
where $\ln x = (\ln x_i^1,\ldots,\ln x_i^{n_i} )^T$, and 
\begin{align}\label{eq:l_ij}
l_{i,j}^{\tau}(x)=\Big(\hat r_i^j + \sum_{k\ne i}\langle \hat q_{ik}^j, &x_k\rangle - \tau\ln x_i^j+\tau\ln\big(1-\sum_{\ell=1}^{n_i}x_i^{\ell}\big) \Big)^2,
\end{align}
where $\hat q_{ik}^j$ is the $j$th row in the matrix $\hat Q_{ik}$ in the definition of the function $f_i^{t}$ (see~\eqref{eq:ut_t1}).
 Let us consider the following multi-objective optimization problem: 
  \begin{align}\label{eq:MOprob}
  \min_{x\in\hD} l_{i,j}^{\tau}(x), \quad i\in[N], j\in[n_i].
 \end{align}
 Then, using Lemma~\ref{lem:solutionsQRE}, the following result holds. 
 \begin{lem}\label{lem:MO}
 There exists a solution to the problem~\eqref{eq:MOprob} with the optimal values $l_{i,j}^{\tau, *} = 0$ for each objective function. Moreover, the solution set for the problem~\eqref{eq:MOprob} is the set of NE in the game $\Gamma^{\tau,t}$.
 \end{lem}
Next, we will focus on the properties of the objective functions $l_{i,j}^{\tau}(x)$ in the optimization problem~\eqref{eq:MOprob}.
\begin{lem}\label{lem:lem1}
For each $i\in[N], j\in[n_i]$, we have $\nabla l_{i,j}^{\tau}(x) = 0$ for some $x\in\hD$ if and only if $x$ minimizes $l_{i,j}^{\tau}(x)$, i.e.  $l_{i,j}^{\tau}(x) =l_{i,j}^{\tau, *}= 0$.
\end{lem} 
\begin{proof}
Using the definition of $l_{i,j}^{\tau}(x)$ from \eqref{eq:l_ij}, we have 
 \begin{align}\label{eq:nabl0}
    &\nabla l_{i,j}^{\tau}(x)\cr
    &=\nabla \Big(\hat r_i^j + \sum_{k\ne i}\langle \hat q_{ik}^j, x_k\rangle - \tau\ln x_i^j + \tau\ln\big(1-\sum_{\ell=1}^{n_i}x_i^{\ell}\big) \Big)^2\cr
    &= 2\Big(\hat r_i^j + \sum_{k\ne i}\langle \hat q_{ik}^j, x_k\rangle - \tau\ln x_i^j + \tau\ln\big(1-\sum_{\ell=1}^{n_i}x_i^{\ell}\big) \Big)\cr
&\qquad\times\boldsymbol{q}_{i,\tau}^j(x_i),
\end{align}
where 
\begin{align}\label{eq:q}
\!\!\!\!\boldsymbol{q}_{i,\tau}^j(x_i)=& \Big(-\frac{\tau}{x_i^j}-
    \frac{\tau}{1-\sum_{\ell=1}^{n_i}x_i^{\ell}},\cr
    & -\frac{\tau}{1-\sum_{\ell=1}^{n_i}x_i^{\ell}},\ldots,-\frac{\tau}{1-\sum_{\ell=1}^{n_i}x_i^{\ell}},\cr
    &\, \hat q_{i1}^j, \ldots,\hat q_{ii-1}^j, \hat q_{ii+1}^j,\ldots,  \hat q_{iN}^j,  0,\ldots, 0\Big)^T\!\!.
\end{align} 
We note that in the definition of $\boldsymbol{q}_{i,\tau}^j(x_i)$ above, we have rearranged the coordinates such that all potentially nonzero elements appear at the beginning of the vector $\boldsymbol{q}_{i,\tau}^j(x_i)$, i.e., the first coordinate corresponds to $\frac{\partial l_{i,j}^{\tau}(x)}{\partial x_i^{j}}$ followed by $\frac{\partial l_{i,j}^{\tau}(x)}{\partial x_i^{-j}}$. 
As $x_i^{\ell}\in [\frac{e^{-1/\tau^2}}{\max_i n_i},1)\ \forall \ell$, we conclude that $\boldsymbol{q}_{i,\tau}^j(x_i)\ne 0$ for any $x\in\hD$. Thus, $\nabla l_{i,j}^{\tau}(x)=0$ iff
\begin{align*}
      \hat r_i^j + \sum_{k\ne i}\langle \hat q_{ik}^j, x_k\rangle - \tau\ln x_i^j + \tau\ln\Big(1-\sum_{\ell=1}^{n_i}x_i^{\ell}\Big) = 0, 
\end{align*}   
which in view of \eqref{eq:l_ij} implies $l_{i,j}^{\tau}(x) =l_{i,j}^{\tau, *}= 0$.
\end{proof}

The above lemma implies that if an optimization procedure finds a point \( x \) such that \( l_{i,j}^{\tau}(x) = 0 \) for all \( i \in [N] \), \( j \in [n_i] \), then \( x \) solves the problem~\eqref{eq:MOprob}. However, since the function \( l_{i,j}^{\tau}(x)\) is non-convex, a gradient-based optimization approach can only guarantee convergence to a stationary point of \( l_{i,j}^{\tau}(x) \). Therefore, as a next step, we proceed to characterize stationary points of \( l_{i,j}^{\tau}(x) \) and establish a connection between a common stationary point of all the functions \( l_{i,j}^{\tau}(x), i \in [N], j \in [n_i]\) and the solution set to the multi-objective optimization problem~\eqref{eq:MOprob}, which, by Lemma~\ref{lem:MO}, equals the set of NE points in the game \( \Gamma^{\tau,t} \).

\begin{definition}
We denote the set of stationary points of the function $l_{i,j}^{\tau}(x)$ on the set $\hD$ by
\begin{align}\nonumber
    \EuScript S^{\tau}_{i,j} = \big\{x\in\hD: \, \langle\nabla l_{i,j}^{\tau}(x), y-x\rangle\ge 0 \quad \mbox{for any }y\in\hD \big\},
\end{align}
and denote the intersection of these stationary point sets by:
\begin{align}\label{eq:statset}
\EuScript S^{\tau}=\cap_{i=1}^N\cap_{j=1}^{n_i} \EuScript S^{\tau}_{i,j}.    
\end{align}   
\end{definition}

Now, we can state the following result. 
\begin{lem}\label{lem:stationary}
For sufficiently small \( \tau > 0 \), the set \( \mathcal{S}^{\tau} \) is nonempty and equals the set of NE points in the game \( \Gamma^{\tau,t} \). 
\end{lem} 
\begin{proof}
Let \( x \) be a NE of the game \( \Gamma^{\tau,t} \). Then, according to Lemma~\ref{lem:MO}, \( l_{i,j}^{\tau}(x) = 0 \) for all \( i \in [N] \), \( j \in [n_i] \). Moreover, Lemma~\ref{lem:lem1} implies that \( \nabla l_{i,j}^{\tau}(x) = 0 \) for all \( i \in [N] \), \( j \in [n_i] \). Thus, \( x \in \mathcal{S}^{\tau} \), and hence, NE points of \( \Gamma^{\tau,t} \) belong to \( \mathcal{S}^{\tau} \).

Now, by contradiction, suppose there exists \( x \in \mathcal{S}^{\tau} \) such that \( x \) is not a NE of the game \( \Gamma^{\tau,t} \). From Lemma~\ref{lem:MO}, there exist \( i' \in [N] \) and \( j' \in [n_{i'}] \) such that \( l_{i',j'}^{\tau}(x) \neq 0 \). Thus, according to Lemma~\ref{lem:lem1}, \( \nabla l_{i',j'}^{\tau}(x) \neq 0 \). However, since \( x \in \mathcal{S}^{\tau}_{i',j'} \), we conclude that \( x \in \partial \hD \). This implies the following two possibilities for the coordinates of \( x \):
\begin{align}\label{eq:1a}
\hspace{-0.2cm}(a) \qquad x_i^j = \frac{e^{-1/\tau^2}}{\max_i n_i+1}\, \mbox{ for some $i\in[N]$, $j\in[n_i]$,}
\end{align}
\begin{align}\label{eq:1b}
\!\!\!\!\!(b) \qquad 1-\sum_{\ell=1}^{n_i} x_i^{\ell} = \frac{e^{-1/\tau^{1.5}}}{\max_i n_i + 1}\, \mbox{ for some $i\in[N]$}.
\end{align}

First, we consider the case (a). Having the relation~\eqref{eq:1a} at place, we can calculate $\nabla l_{i,j}^{\tau}(x)$ using \eqref{eq:nabl0} to obtain: 
\begin{align}\label{eq:nabl}
\nabla l_{i,j}^{\tau}(x)&=2 \Bigg(\hat r_i^j + \sum_{k\ne i}\langle \hat q_{ik}^j, x_k\rangle +\tau \ln (\max_i n_i+1)\cr
&\qquad+ \frac{1}{\tau} +\tau\ln\Big(1-\sum_{\ell=1}^{n_i}x_i^{\ell}\Big)\Bigg)\boldsymbol{q}_{i,\tau}^{j}(x_i),
\end{align}
where $\boldsymbol{q}_{i,\tau}^{j}(x_i)$ is given by \eqref{eq:q}.   


Next, we demonstrate that for a sufficiently small $\tau>0$ the multiplier $\hat r_i^j + \sum_{k\ne i}\langle \hat q_{ik}^j, x_k\rangle  + \frac{1}{\tau} + \tau \ln (\max_i n_i+1)+ \tau\ln\left(1-\sum_{\ell=1}^{n_i}x_i^{\ell}\right)$ in the definition of $\nabla l_{i,j}^{\tau}(x)$ in~\eqref{eq:nabl} is positive for any $x\in\hat\Delta$.
According to the definition of the set $\tilde{\Delta}$ (see~\eqref{eq:extended_Di} and~\eqref{eq:extended_D}), we have
\begin{align}\label{eq:2}
1-\sum_{\ell=1}^{n_i}x_i^{\ell}\ge \frac{e^{-1/\tau^{1.5}}}{\max_i n_i+1}.
\end{align}
Thus, 
\begin{align}\label{eq:3}
 \tau\ln\Big(1-\sum_{\ell=1}^{n_i}x_i^{\ell}\Big)\ge -\frac{1}{\tau^{0.5}}-\tau \ln (\max_i n_i+1).
\end{align}
There exists sufficiently small $\tau$  such that for any $x\in\hD$
\begin{align}\label{eq:4}
\hat r_i^j + \sum_{k\ne i}\langle \hat q_{ik}^j, x_k\rangle + \frac{1}{\tau} -\frac{1}{{\tau^{0.5}}}>0.
\end{align}
Using relations~\eqref{eq:nabl},~\eqref{eq:3}, and~\eqref{eq:4}, we conclude that the vectors $\nabla l_{i,j}^{\tau}(x)$ and $\boldsymbol{q}_{i,\tau}^{j}(x_i)$ are the same up to a positive multiplier. Now, consider a specific $y\in\hD$ that is defined by $y_i^j>x_i^j=\frac{e^{-1/\tau^2}}{\max_i n_i+1}$, $\sum_{\ell=1}^{n_i} y_i^{\ell}  \ge \sum_{\ell=1}^{n_i} x_i^{\ell}$, $y_{m}^{\ell} =x_{m}^{\ell}$ for all $m\ne i$ and $\ell\in[n_{m}]$, and note that such a feasible $y\in\hD$ always exists. Then,   
\begin{align}\label{eq:<q,y-x>}
\hspace{-0.5cm}\langle \boldsymbol{q}_{i,\tau}^{j}(x_i), y-x\rangle&=-\frac{\tau}{x_i^j}(y_i^j-x_i^j)\cr
&-\frac{\tau}{1-\sum_{\ell=1}^{n_i}x_i^{\ell}} \Big(\sum_{\ell=1}^{n_i}y_i^{\ell}-\sum_{\ell=1}^{n_i}x_i^{\ell}\Big),
\end{align}
which is negative due to the choice of $y$.
This shows that $\langle\nabla l_{i,j}^{\tau}(x), y-x\rangle<0$ for some $y\in\hD$. Thus, $x\notin \mathcal{S}_{i,j}^{\tau}$, which contradicts the assumption $x\in \EuScript S^{\tau}$. 

Next, we will consider case (b). As~\eqref{eq:1b} holds, there exists $j\in[n_i]$ such that 
\begin{align}\label{eq:1b1}
    x_i^j \ge \frac{1}{n_i} - \frac{e^{-1/\tau^{1.5}}}{n_i(\max_i n_i + 1)}\ge \frac{1}{n_i+1}.
\end{align}
By calculating $\nabla l_{i,j}^{\tau}(x)$ using \eqref{eq:nabl0} under case (b), we have
\begin{align}\label{eq:nabl1b}
    \nabla l_{i,j}^{\tau}(x)&= 2\Bigg(\hat r_i^j + \sum_{k\ne i}\langle \hat q_{ik}^j, x_k\rangle - \tau\ln x_i^j -\frac{1}{\tau^{0.5}} \cr
    &\qquad\qquad- \tau\ln(\max_in_i+1) \Bigg)\boldsymbol{q}_{i,\tau}^{j}(x_i),
\end{align}
where $\boldsymbol{q}_{i,\tau}^j(x_i)$ is given by~\eqref{eq:q}. We note that the multiplier in the expression \eqref{eq:nabl1b} is negative for any $x \in \tilde{\Delta}$, given an appropriate small choice of $\tau$. The reason is that using \eqref{eq:1b1},
\begin{align}\label{eq:tau2}
&\hat{r}_i^j + \sum_{k \neq i} \langle \hat{q}_{ik}^j, x_k \rangle - \tau \ln x_i^j - \frac{1}{\tau^{0.5}} - \tau \ln (\max_i n_i + 1)\cr
&\leq \hat{r}_i^j + \sum_{k \neq i} \langle \hat{q}_{ik}^j, x_k \rangle - \frac{1}{\tau^{0.5}}<0.
\end{align}
where the second inequality holds for sufficiently small $\tau$. Therefore, $\nabla l_{i,j}^{\tau}(x)$ and $\boldsymbol{q}_{i,\tau}^j(x_i)$ are the same vectors up to a negative multiplier.

Finally, using a similar argument as in case (a), let us consider a specific vector $y \in \hat{\Delta}$ defined by $y_i^j < x_i^j$, $\sum_{\ell=1}^{n_i} y_i^\ell \le \sum_{\ell=1}^{n_i} x_i^\ell$, $y_m^\ell = x_m^\ell$ for all $m \neq i$ and $\ell \in [n_m]$. Then, using \eqref{eq:<q,y-x>}, the choice of $y$, and the fact that $-\frac{\tau}{x_i^j} < 0$ and $-\frac{\tau}{1 - \sum_{\ell=1}^{n_i} x_i^\ell} < 0$, we obtain $\langle \boldsymbol{q}_{i,\tau}^j(x_i), y - x \rangle > 0$, which in turn implies $\langle \nabla l_{i,j}^{\tau}(x), y - x \rangle < 0$. Therefore, $x \notin \mathcal{S}_{i,j}^{\tau}$, which contradicts the assumption that $x \in \EuScript{S}^{\tau}$.

The obtained contradictions in both cases (a) and (b) imply that $\EuScript{S}^{\tau}$ is a subset of the NE points in the game $\Gamma^{\tau,t}$, which completes the proof.
\end{proof}    

Lemmas~\ref{lem:epsApp},~\ref{lem:solutionsQRE}, and~\ref{lem:stationary} imply the following proposition. 
\begin{prop}\label{prop:1}
 Given a fixed $\epsilon > 0$, let $\tau \le \frac{\epsilon}{\max_i\{\ln (n_i+1)\}}$. There exists a sufficiently small $\epsilon$ such that any point $x \in \EuScript{S}^{\tau}$ is an $\epsilon$-NE in the game $\Gamma$, where $\EuScript{S}^{\tau}$ represents the set of common stationary points of the objective functions $l_{i,j}^{\tau}$ over $\hat\Delta$ in the multi-objective optimization problem~\eqref{eq:MOprob}. 
\end{prop}

\begin{remark}\label{rem:tau}
One can obtain an explicity valid range for $\epsilon$ in Proposition~\ref{prop:1}. Indeed, by combining the relations~\eqref{eq:tau1},~\eqref{eq:4}, and~\eqref{eq:tau2} for a sufficiently small choice of $\tau$, we conclude that the statement of Proposition~\ref{prop:1} holds if
\begin{align}\nonumber
    \tau < \tau_{\text{max}} := \min\left\{\frac{1}{\delta_f}, \frac{1}{\delta_f^2}, \frac{1}{4}, \frac{1}{R^2}\right\},
\end{align}
where $R = \max_{i,j,x_k \in \hD_k}\big\{\hat{r}_i^j + \sum_{k \neq i} \langle \hat{q}_{ik}^j, x_k \rangle\big\}$, and $\hat{q}_{ik}^j$ and $\hat{r}_i^j$ are defined in \eqref{eq:ut_t1}. Thus, $\epsilon$ should be chosen as follows:
\[
\epsilon < \tau_{\text{max}} \max_i \left\{\ln(n_i + 1)\right\}.
\]
\end{remark}

\section{Distributed Solution Approaches}\label{sec:IV}
Equipped with the results from the previous section, we now present two heuristic distributed algorithms in this section to obtain an $\epsilon$-NE solution. Of course, in view of Remark \ref{rem:compl}, we do not expect our algorithms to always find an $\epsilon$-NE efficiently. However, as we will discuss, the proposed algorithms exhibit some desirable convergence properties. 

To that end, we assume throughout this section the existence of a weighted undirected communication graph $G_{w}([N], \EuScript{A})$ connecting the players. The set of nodes $[N]$ corresponds to $N$ players, and the arc $(i,j) \in \EuScript{A}$ exists if there is a link between players $i$ and $j$. We make the following standard assumption regarding the communication graph $G_{w}$ \cite{Bianchi, NedOzPar}.

\begin{assum}\label{assum:comm}
The underlying undirected comm  unication graph $G_{w}([N], \EuScript{A})$ is connected. The associated non-negative, symmetric mixing matrix $W = [w_{ij}] \in \mathbb{R}^{N \times N}$ defines the weights on the undirected arcs such that $w_{ij} > 0$ if and only if $\{i,j\} \in \EuScript{A}$, and $\sum_{j=1}^{N} w_{ij} = 1 - \beta$ for all $i \in [N]$ and some $\beta \in (0,1)$.
\end{assum}

\subsection{A Distributed Projected Gradient Descent Algorithm}

Our first proposed algorithm uses the result of Proposition~\ref{prop:1}, by focusing on solving the optimization problem~\eqref{eq:MOprob}. First, following the work~\cite{Bianchi}, we aim to adapt a distributed algorithm to calculate stationary points of the function 
\begin{align}\label{eq:l}
  l^{\tau}(x)= \frac{1}{n}\sum_{i\in[N], j\in[n_{i}]}l_{i,j}^{\tau}(x).
\end{align}

During the optimization procedure, we let each player $i \in [N]$ process $n_i$ estimates of the joint action $x \in \hat{\Delta}$, which are represented by the vectors $\tx_{i}^j(t) \in \hat{\Delta}, j \in [n_i]$. Specifically, let each player $i$ communicate its estimate vectors $\tx_{i}^j(t)$, $j \in [n_i]$, via the communication step: 
\begin{align}\nonumber
       x_{i}^j(t)=\sum_{i'\in[N], j'\in[n_{i'}]} w^{j,j'}_{i,i'}\tx_{i'}^{j'}(t)\ \ \ \forall j\in [n_i],
\end{align}
where the weight parameters $w^{j,j'}_{i,i'}$ are chosen as follows: 
\begin{align}\label{eq:weights}
 w^{j,j'}_{i,i'} =  
 \begin{cases}
   w_{ij}, &\mbox{ if $i\ne i'$ and $j=j'=1$},\\ 
   \frac{\beta}{n_i}, \,\,&\mbox{ if $i = i'$}, \\
    0, \,\,\,\,\,\,&\mbox{ otherwise}.
\end{cases} 
\end{align}
Then, each player $i$ updates its estimates $\tx_{i}^j(t)$, $j\in[n_i]$ at the next time step using projected gradient descent method with an appropriate choice of step-size (see \eqref{eq:da}). The overall algorithm is summarized in Algorithm \ref{alg:GD}.

\begin{algorithm}[H]\caption{Projected Gradient Descent for Player $i$}\label{alg:GD}
{\bf Input:} Each player $i$ process $n_i$ estimates of the joint action with initial values $\tilde{x}_i^j(0), j\in [n_i]$; a stepsize sequence $\{\gamma_t\}$.  

\noindent
{\bf For} $t=1,2,\ldots$, player $i$ performs the following steps:

\begin{itemize}
\item Compute the averaged estimates using weights \eqref{eq:weights} as
\begin{align}\label{eq:da_c}
x^j_i(t)&=\sum_{i',j'} w^{j,j'}_{i,i'}\tilde{x}^{j'}_{i'}(t)\ \ \forall j\in [n_i].
\end{align}
\item Update the estimate vectors at the next time step by
\begin{align}\label{eq:da}
    &\tx_i^j(t+1)= \Proj_{\hD}\left[x_i^j(t) - \gamma_t \nabla l_{i,j}^{\tau}(x_i^j(t))\right]\ \forall j.
\end{align}
\end{itemize}
\end{algorithm}

We note that based on the weights $w^{j,j'}_{i,i'}$ defined in \eqref{eq:weights}, one can define an extended communication graph defined on $n = \sum_{i=1}^N n_i$ nodes, with the symmetric mixing matrix $W'$ defined by its elements in~\eqref{eq:weights}.
Then, the procedure~\eqref{eq:da_c}-\eqref{eq:da} is a deterministic version of the projected gradient-based optimization procedure  in~\cite{Bianchi} (see the procedure~(2) therein) for a distributed approach solving the non-convex problem: $\min_{x\in\hat\Delta}\ l^{\tau}(x)$, where $l^{\tau}(x)$ is defined in~\eqref{eq:l}. Let $\EuScript L^{\tau}$ denote the set of the stationary points of the function $l^{\tau}(x)$ on the set $\hD$, i.e.,
\begin{align}\nonumber
    \EuScript L^{\tau} = \big\{x\in\hD: \, \langle\nabla l^{\tau}(x), y-x\rangle\ge 0 \quad \mbox{for any }y\in\hD \big\}.
\end{align}
Now, we have the following result from \cite{Bianchi}.
\begin{theorem}\label{th:bianchi}[Deterministic version of \cite[Theorem 1]{Bianchi}]
Let Assumption~\ref{assum:comm} hold. Then, under the choice of communication weights in~\eqref{eq:weights}, and given that $\sum_{t=1}^{\infty} \gamma_t = \infty$ and $\sum_{t=1}^{\infty} \gamma_t^2 < \infty$, as $t \to \infty$, all $\tx_{(i,j)}(t), i \in [N], j \in [n_i]$, updated by the procedures~\eqref{eq:da_c}-\eqref{eq:da}, converge to the same point in the set $\EuScript{L}^{\tau}$.
\end{theorem}

Having the above result in hand and taking into account the existence of a common stationary point $x \in \EuScript{S}^{\tau}$ of the functions $l_{i,j}^{\tau}$ on the set $\hD$, i.e., $\EuScript{S}^{\tau} \neq \emptyset$ and $\EuScript{S}^{\tau} \subseteq \EuScript{L}^{\tau}$ (see~\eqref{eq:statset} and Lemma~\ref{lem:stationary}), if we let each player follow Algorithm \ref{alg:GD} with some initial estimates $\tx^j_{i}(0)$ and use the sequence of step sizes $\{\gamma_t\}$ satisfying Theorem~\ref{th:bianchi}, then the iterates are guaranteed to converge to a common point in the set $\EuScript{L}^{\tau}$. Thus, this point will either be from the set $\EuScript{S}^{\tau}$ or from the set $\EuScript{L}^{\tau} \setminus \EuScript{S}^{\tau}$. In the former case, according to Proposition~\ref{prop:1}, an $\epsilon$-NE is achieved. However, in the latter case, one should rely on some heuristic algorithm to further steer the dynamics to a point in $\EuScript{S}^{\tau}$. We should emphasize again that, according to the complexity result of Remark~\ref{rem:compl}, one cannot expect an efficient calculation of a point from $\EuScript{S}^{\tau}$ unless further assumptions are imposed on the game. For instance, one heuristic could be to check the condition in Lemma~\ref{lem:solutionsQRE} after some sufficiently large number of iterates $T$. If there exists at least one player $i \in [N]$ such that $\|\nabla l_{i,j}^{\tau}(\tx_{i}^j(T))\| > \epsilon_0$ for some $j \in [n_i]$, given some fixed threshold $\epsilon_0 > 0$, the procedure will be restarted with a random selection of initial estimates $\tx_{i}^{'j}(0) \neq \tx_{i}^j(0)$ for $i \in [N], j \in [n_i]$.

\subsection{A Distributed Fixed-Point Approximation Algorithm}

Our second proposed algorithm relies on the result of Lemma \ref{lem:solutionsQRE} to directly approximate the fixed point of the best response dynamics in the game $\Gamma^{\tau}$. More precisely, the algorithm aims to solve the system of equations $\nabla l^{\tau}_{ij}(x) = 0$ for each $i,j$ in a distributed manner. After each update, the players average their estimates with those of others to ensure they converge to the same point in the limit. The overall procedure is summarized in Algorithm \ref{alg:fix}. As before, we use the notation $\tilde{x}_i^{j}$ to denote the $j$th estimate of the joint action that player $i$ processes as time progresses, whereas $\tilde{x}_i^{j,\ell}$, $\ell \in [n_i]$, represents the local coordinates updated by player $i$ in the estimate vector $\tilde{x}_i^{j}$.

\begin{prop}
Let Assumption~\ref{assum:comm} hold. If the dynamics of Algorithm \ref{alg:fix} converge, they converge to an $\epsilon$-NE of the game $\Gamma$.    
\end{prop}
\begin{proof}
Using \eqref{eq:nabl0}, one can solve for the solutions of $\nabla l^{\tau}_{ij}(x)= 0$ to obtain
\begin{align}\label{eq:fix-point}
x_i^j = \frac{\exp\left(\frac{1}{\tau}(\hat{r}_i^j + a^j_i x_{-i})\right)}{1 + \sum_{\ell} \exp\left(\frac{1}{\tau}(\hat{r}^{\ell}_i + a^{\ell}_i x_{-i})\right)} \quad \forall i,j,    
\end{align}
where $a_i^j$ is shorthand for the vector $(\hat{q}^{j}_{ik}, k \neq i)$. Moreover, by Lemma \ref{lem:solutionsQRE}, a solution to this system of nonlinear equalities will be an $\epsilon$-NE.  

Now, if the dynamics of Algorithm \ref{alg:fix} converge to some $x$ for all $i\in[N]$, $j\in[n_i]$, i.e., $\lim_{t \to \infty} \tilde{x}^j_i(t) = x \ \forall i,j$, then, due to the averaging step in \eqref{eq:average-fixed} and Assumption \ref{assum:comm}, the estimates of all the players must converge to the same point. That is, $\lim_{t \to \infty} \tilde{x}^j_i(t) = x\ \forall i,j$, and in particular, $\lim_{t \to \infty} x^j_i(t) = x\ \forall i,j$. Finally, by taking the limit of both sides of the update dynamics in \eqref{eq:fixed-update} and using the continuity of the functions, one can see that the limit point $x$ must satisfy \eqref{eq:fix-point}, and hence it must be an $\epsilon$-NE.  
\end{proof}

\begin{algorithm}[t]\caption{Fixed-Point Approximation for Player $i$}\label{alg:fix}
{\bf Input:} Each player $i$ process $n_i$ estimates of the joint action with initial values $\tilde{x}_i^j(0)\in \hD, j\in [n_i]$.  

\noindent
{\bf For} $t=1,2,\ldots$, player $i$ performs the following steps:

\begin{itemize}
\item Compute the averaged estimates using weights \eqref{eq:weights} as
\begin{align}\label{eq:average-fixed}
x^j_i(t)&=\sum_{i',j'} w^{j,j'}_{i,i'}\tilde{x}^{j'}_{i'}(t)\ \ \forall j\in [n_i].
\end{align}
\item Update the local coordinates in the estimate vectors at the next time step by
\begin{align}\label{eq:fixed-update}
\!\!\!\!\tilde{x}^{j,\ell}_i(t+1)\!=\!\frac{\exp(\frac{1}{\tau}(\hat{r}_i^{\ell}+a^j_i x_{-i}(t)))}{1+\sum_{\ell} \exp(\frac{1}{\tau}(\hat{r}^{\ell}_i+a^{\ell}_ix_{-i}(t)))}\ \forall j,\ell\in[n_i],
\end{align}
where $a_i^{\ell}=(\hat{q}^{\ell}_{ik}, k\neq i)$.
\end{itemize}
\end{algorithm}

Finally, we note that while the convergence of Algorithm \ref{alg:fix} is not always guaranteed, since the sequence $\{\tilde{x}_j^i(t)\}$ belongs to a closed and bounded set, it has a limit point. Therefore, given existence of a common limit point, at least a subsequence of the iterates generated by Algorithm \ref{alg:fix} will converge to an $\epsilon$-NE.

\section{Simulation Results}\label{sec:V}
We consider a two-player, non-zero sum game, which can arise, for example, in a conflict situation on roads where two vehicles, the players in the game, interact at a non-signalized intersection. In such a situation, each player is assumed to choose between two actions: either ``go" or ``stop." The payoff matrices of the players are presented in Table~\ref{tab:game}.

\begin{table}
\caption{Payoff/utility matrices.}
\begin{center}
\begin{tabular}{ | m{1cm} | m{1cm}| m{1cm} | } 
  \hline
  & go & stop \\ 
  \hline
  go  & (-6,-6) & (1,-4) \\ 
  \hline
  stop & (-4,1) & (0,0) \\ 
  \hline
\end{tabular}\label{tab:game}
\end{center}
\end{table}

\begin{figure*}[!h]
	\centering
	\begin{minipage}[t]{0.45\textwidth}
		\centering
		\includegraphics[width=\textwidth]{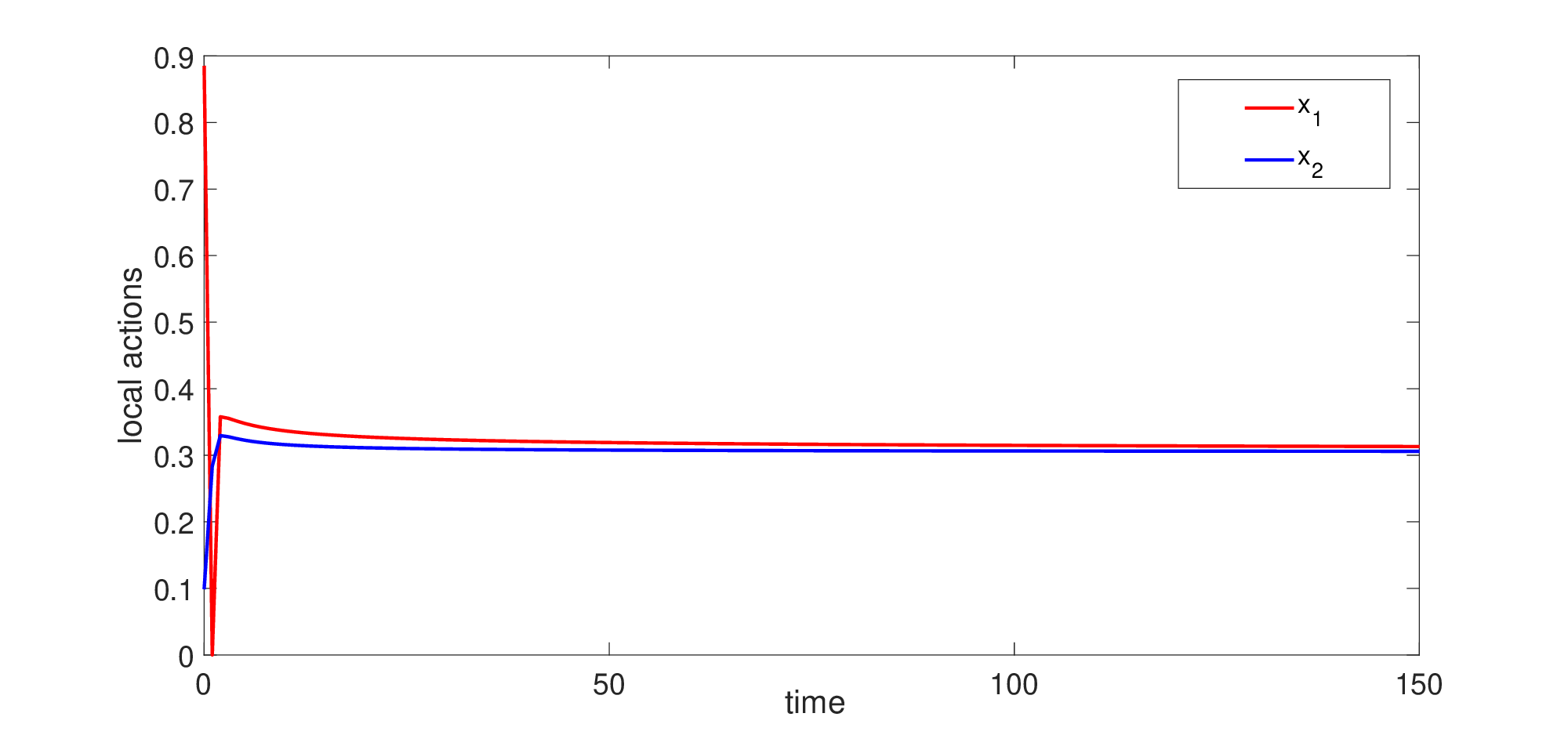}
		\label{fig:d1point}
	\end{minipage}
	\hfill
	\begin{minipage}[t]{0.45\textwidth}
		\centering
		\includegraphics[width=\textwidth]{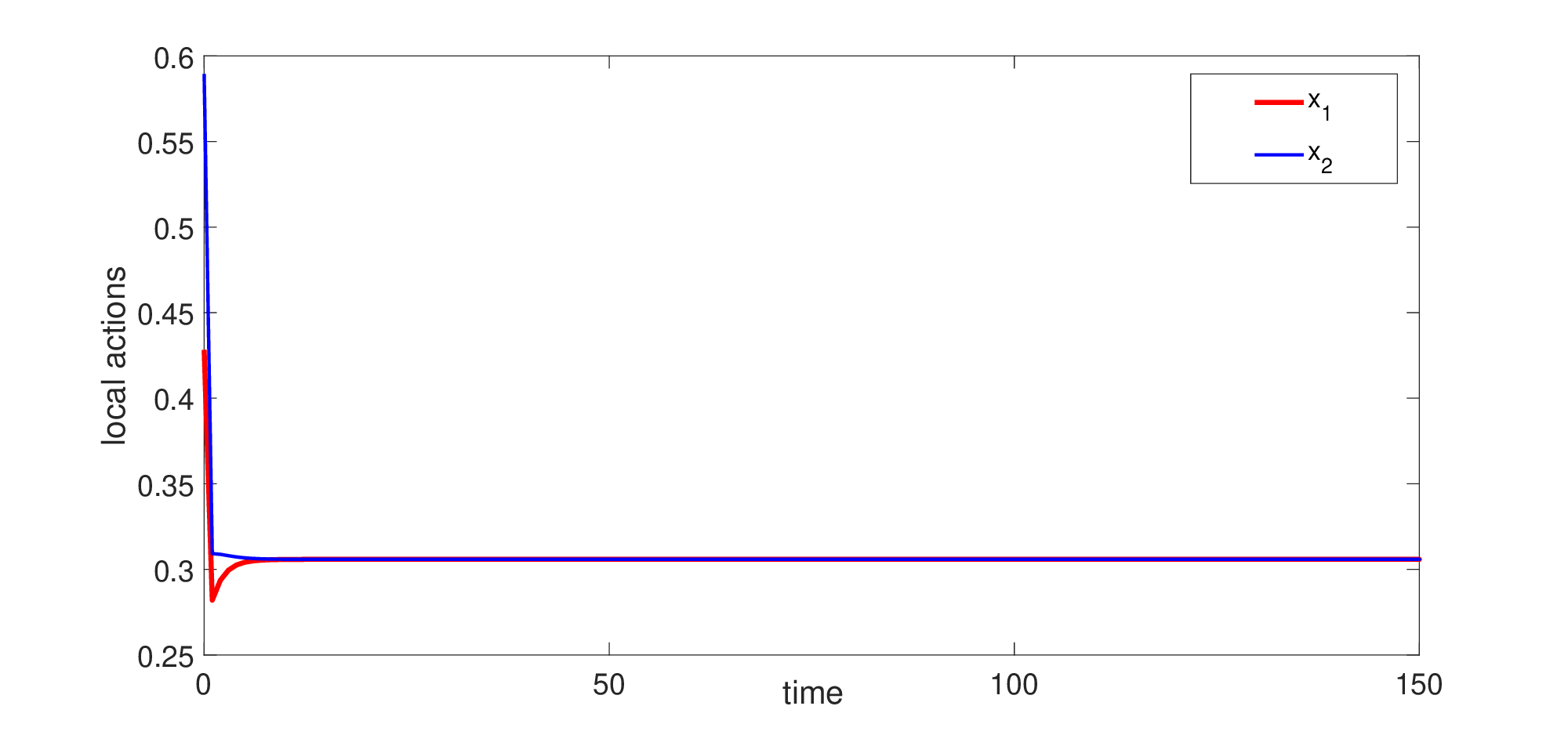}
			\label{fig:alg}
	\end{minipage}
	\caption{Local actions $x_1(t)$ and $x_2(t)$ updated within the estimate vectors $\tilde{x}_1(t) = (x_1(t),\tilde{x}_1^2(t))$ and $\tilde{x}_2(t) = (\tilde{x}_2^1(t), x_2(t))$ in the run of Algorithm~1 (left) and of Algorithm~2 (right).}
	\label{fig:distance}
\end{figure*}

The direct calculation of the regularized utility functions in the mixed strategies leads to the following multi-objective optimization problem: 
\begin{align*}
    &l^{\tau}_1(x_1,x_2) = (-3x_2-4-\tau \ln x_1 + \tau\ln(1-x_1))^2\to\min,\cr
    &l^{\tau}_1(x_1,x_2) = (-3x_1-4-\tau \ln x_2 + \tau\ln(1-x_2))^2\to\min,\cr
    &\mbox{s.t. $(x_1,x_2)\in\hD$},
\end{align*}
where $\hD$ is defined according to~\eqref{eq:extended_D}, and $x_1$ and $x_2$ are the probabilities of taking the action ``go" for players $i=1$ and $i=2$, respectively. Player 1 controls the function $l^{\tau}_1$, whereas Player 2 has access to the function $l^{\tau}_2$. In this setting, we run Algorithms 1 and 2. In the simulations, we scale the utility functions by a factor of 6, set $\gamma_t = \frac{5}{t}$, $w_{ij}=0.5$ for all $i,j$, and randomly initialize the two-dimensional vector of joint action estimates, i.e., $\tilde{x}_1(0)$ and $\tilde{x}_2(0)$ for both players. The simulation results are shown in Figure~\ref{fig:alg}. As we can see, both algorithms converge to a neighborhood of the mixed NE $x_1 = x_2 = \frac{1}{3}$ after a few iterations.

\section{Conclusion}\label{sec:VI}
This work formulates a multi-objective optimization approach to approximate Nash equilibria in a specific class of convex games over the simplex. The main advantage of this formulation lies in the fact that the common stationary points are zeros of the objective functions as well as their gradients, and they coincide with the approximate Nash equilibria in the game under consideration. These characteristics of the approximate Nash equilibria allow for the application of heuristics for their computation. Future work will focus on more sophisticated computation techniques with theoretical guarantees regarding subsequence convergence to an approximate Nash equilibrium.

\bibliographystyle{plain}
\bibliography{CrowdSGames_ref}

\end{document}